\documentclass{amsart}
\usepackage[all,knot]{xy}
\xyoption{arc} 
\usepackage{amsfonts}
\usepackage{array}
\usepackage{euscript}
\usepackage{epsfig, psfrag, epic}
\usepackage{graphicx}
\usepackage{xy}
\usepackage{multirow}
\usepackage{amsmath}
\usepackage{subfigure}

\title{Fox-Neuwirth cell structures and the cohomology of symmetric groups}

\newtheorem{thm}{Theorem}[section]
\newtheorem{cor}[thm]{Corollary}
\newtheorem{lem}[thm]{Lemma}
\newtheorem{prop}[thm]{Proposition}

\theoremstyle{definition}
\newtheorem{defn}[thm]{Definition}

\newcommand{\sq}{{Sq}}
\newcommand{\indec}{{\rm Indec}}
\newcommand{\amod}{{\mathcal A}-{\rm Mod}}
\newcommand{\aalg}{{\mathcal A}-{\rm Alg}}
\newcommand{\al}{{\rm Alg}}

\newcommand{\into}{\hookrightarrow}

\newcommand{\R}{{\mathbb R}}
\newcommand{\Z}{{\mathbb Z}}

\newcommand{\ccup}{\smile}
\newcommand{\tr}{\odot}

\newcommand{\x}{{\mathbf x}}
\newcommand{\y}{{\mathbf y}}
\newcommand{\z}{{\mathbf z}}
\newcommand{\eh}{{\mathbf a}}
\newcommand{\bb}{{\mathbf b}}

\newcommand{\si}{{\mathcal S}}
\newcommand{\sgn}{{\rm sgn}}
\newcommand{\fn}{{\rm FN}}
\newcommand{\F}{{\mathbb F}}

\newcommand{\Symm}{\text{Symm}}
\newcommand{\Sh}{\text{Sh}}

\newcommand{\refT}[1]{Theorem~\ref{T:#1}}

\newcommand{\refP}[1]{Proposition~\ref{P:#1}}

\newcommand{\refF}[1]{Figure~\ref{F:#1}}

\newcommand{\Conf}[2]{\textrm{Conf}_{#1}(\mathbb{R}^{#2})}
\newcommand{\UConf}[2]{\overline{\textrm{Conf}}_{#1}(\mathbb{R}^{#2})}

\begin{document}

\author[C. Giusti]{Chad Giusti}
\address{Mathematics Department, 
Willamette University}
\email{cgiusti@willamette.edu}
\author[D. Sinha]{Dev Sinha}
\address{Mathematics Department, 
University of Oregon}
\email{dps@math.uoregon.edu}

\begin{abstract}
We use the Fox-Neuwirth cell structure for one-point compactifications of configuration spaces as the starting 
point for understanding our recent calculation of the mod-two cohomology of symmetric groups.  We then use that
calculation to give short proofs of classical results on this cohomology due to Nakaoka and to Madsen.
\end{abstract}

\maketitle

\section{Introduction}

Group cohomology touches on a range of subjects within algebra and topology.  It is thus amenable to study
by a similar range of techniques, as we have found in our recent work on the cohomology of symmetric groups.
The key organizational tool is the algebraic notion of a Hopf ring.   But it is the geometry of bundles and 
that of the corresponding
models for classifying spaces which have guided us, and in particular led to  rediscovery of the 
Hopf ring structure.

We consider cohomology of symmetric groups as giving characteristic classes for finite-sheeted covering maps.
Equivalently, we consider the configuration space models for their classifying spaces.  There is a strong analogy with characteristic classes for vector bundles and the geometry of Grassmannians.   In the
vector bundle setting, we embed
a bundle over a manifold base space in a trivial bundle and define characteristic classes which are Poincar\'e dual to the
locus where the fiber intersects the standard flag in subspaces with a prescribed set of dimensions.  These cohomology classes
are pulled back from classes represented by Schubert cells.  We may use the isomorphism  $O(1) = \si_{2}$, to translate to covering
spaces.   Embed a two-sheeted covering of a manifold once again in a trivial vector bundle.  Then the 
characteristic classes, which are all powers of the first Stiefel-Whitney class,
are Poincar\'e dual to the locus where the two points in a fiber share their first $n$ coordinates.
We will see that in general characteristic classes of finite sheeted 
covering spaces are defined by loci where the points in the fiber are (nested) partitioned
into collections of points which share prescribed coordinates.

A critical component of the the theory of characteristic classes for vector bundles is the Schubert cell structure
on Grassmannians.
We begin this paper by developing an analogue for configuration spaces, namely 
the Fox-Neuwirth cell structure on their one-point compacitifications.  While this cell structure in the two-dimensional setting 
dates back fifty years  \cite{FoNe62} and was developed in all dimensions mod-two by Vassiliev \cite{Vass92},
we present what is to our
knowledge the first treatment of the differential integrally.  This cell structure is ``Alexander dual'' to one on configuration spaces
themselves which has enjoyed a renaissance of interest 
lately, being studied in a number of contexts, namely by Tamaki \cite{Tama11.1, Tama11.2} related to work on the iterated cobar 
construction \cite{Tama94, Tama06}, by Ayala and Hepworth \cite{AyHe12} in the context of higher category theory and
by  Blagojevi\'{c} and  Ziegler \cite{BlZi12} in the context of partitioning convex bodies.
After presenting the cell structure, we compute the mod-two 
cohomology of $B\si_{4}$, giving cochain representatives.  

In order to understand 
arbitrary symmetric groups, we find it essential to consider them all together, much as in homology where their direct sum 
has an elegant  description as a ``Q-ring'' \cite{BiJo97}.  In cohomology the relevant structure is that of a Hopf ring,
due to Strickland and Turner \cite{StTu97}.
This was the key organizing structure of our work with Salvatore in \cite{GSS09}.
Our construction of a model for this Hopf ring structure at the level of Fox-Neuwirth cochains is new.
After giving a Hopf ring presentation in terms of generators and relations, we give a graphical presentation of the resulting
monomial basis which we call the skyline diagram basis, recapping the major results from our paper \cite{GSS09} and 
connecting them with Fox-Neuwirth cochain representatives first identified by Vassiliev \cite{Vass92}.
We use the skyline diagram presentation to revisit  the structure of the 
cohomology of $S_{\infty}$ as an algebra over the Steenrod algebra, and in particular to see the Dickson 
algebras as an associated graded.

\tableofcontents

\section{The Fox-Neuwirth cochain complexes} \label{S:cellstructure}

In this section we  decompose configuration spaces into open cells, analogously to how one decomposes Grassmannians into Schubert
cells.
Consider $\mathbb{R}^{m}$ with its standard basis, and let $\Conf{n}{m}$ denote the space of configurations of $n$ distinct, labelled
points in $\R^m$ topologized as a subspace of $\R^{mn}$.  The configuration space
$\Conf{n}{m}$ admits a free, transitive action of $\si_n$, the symmetric group on $n$ letters, by permutation of the labels.
 We denote by $\UConf{n}{m}$ the quotient  $\Conf{n}{m}/\si_n$, which is the space of unlabeled configurations of $n$ points in $\R^m$.

The standard embeddings $\R^m \into \R^{m+1}$ give rise to a canonical directed system of configuration spaces
$\Conf{n}{2} \into \Conf{n}{3} \into \cdots$ whose maps are all equivariant with respect to the $\si_n$ action.  We denote the limit by $\Conf{n}{\infty}$. A map from a sphere of dimension less than  $m-1$ to $\Conf{n}{m}$ can be first null-homotoped
in $\R^{mn}$, but then by general position be null-homotoped in $\Conf{n}{m}$ itself.  Thus 
$\Conf{n}{m}$ is $m-2$ connected, so $\Conf{n}{\infty}$ is weakly contractible.  
Moreover $\Conf{n}{\infty}$ inherits a free action of $\si_n$, so it is a model for 
$E\si_n$. Thus $\UConf{n}{\infty} \simeq B\si_n$.   Indeed, $B \si_{n}$ classifies finite-sheeted covering spaces, and by
 embedding such a bundle over some paracompact base in a trivial Euclidean bundle we can define the 
 classifying map to  $\UConf{n}{\infty}$.

\subsection{Fox-Neuwirth cells} We now describe ``open''
cellular decompositions  of $\Conf{n}{m}$ and $\UConf{n}{m}$ which then define CW-structures on their
one-point compactifications.  
 These cellular decompositions are due to Fox and Neuwirth \cite{FoNe62} when $m=2$, were considered by 
 Vassiliev in higher dimensions \cite{Vass92},
 and have also been considered in the context of $E_{m}$-operads by Berger \cite{Berg96, Berg97}.  They are in some sense
Alexander dual to the Milgram decompositions of configuration spaces \cite{Milg66} (and more generally the Salvetti complex \cite{Salv87})
and the realizations of $\theta$-categories of Joyal \cite{Joya97}.
In the limit these cellular decompositions give rise to a cochain complex to compute the cohomology of symmetric groups - see \refT{fnmodel}.

The cell structure is based on the dictionary ordering of points in $\R^{m}$ using standard coordinates, which we
denote by $<$.  This ordering gives rise to an ordering of points in a configuration.

\begin{defn}
A depth-ordering is a series of labeled inequalities $i_{1} <_{a_{1}}  i_{2} <_{a_{2}} i_{3} <_{a_{3}} \cdots <_{a_{n-1}} i_{n}$, where 
the $a_{i}$ are non-negative integers and the set of $\{ i_{k} \}$ is exactly $\{1, \ldots, n\}$.

For any depth-ordering $\Gamma$ define $\Conf{\Gamma}{m}$ to be the collection of all configurations $(x_1, x_2,\dots,x_n) \in \Conf{n}{m}$ 
such that for any $p$, $x_{i_{p}} < x_{i_{p+1}}$ and their first $a_{p}$ coordinates are equal while their $(a_{p} + 1)$st coordinate must differ.  
\end{defn}

The subspace $\Conf{\Gamma}{m}$ is empty if any $a_{i}$ is greater than or equal to $m$.

\begin{thm}[after Fox-Neuwirth]\label{T:cells}
For any $\Gamma$ the subspace $\Conf{\Gamma}{m}$ is homeomorphic to a Euclidean ball of dimension $mn - \sum a_{i}$.
The images of the $\Conf{\Gamma}{m}$ are the interiors of cells in an equivariant 
CW structure on the one-point compactification  $\Conf{n}{m}^{+}$.
\end{thm}

Because this cell structure is equivariant, it descends to one for $\UConf{n}{m}^{+}$, 
where the cells are indexed only 
by the $a_{i}$.  But to understand the boundary structure, we need to work with ordered configurations.  While these
cells are very easy to name, their boundary structure is intricate.  Consider the cell  $\Conf{\Gamma}{m}$ with 
$\Gamma =1 < _{1} 2 <_{0} 3$.  It is immediate to see the cell labeled by $1 < _{2} 2 <_{0} 3$ on the boundary of $\Gamma$, as the 
 inequality between the second coordinates of $x_{1}$ and $x_{2}$ degenerates to 
an equality.   But the cell labeled by $2 <_{2} 1 <_{0} 3$ is also on the boundary, since the third coordinates of $x_{1}$ and $x_{2}$ are unrestricted as the second coordinates become equal.  There are even more possibilities when  the inequality between the first coordinates of $x_{2}$ (and thus $x_{1}$ as well) and $x_{3}$ degenerates to an equality.  The first inequality,  based on the second
coordiates  of $x_{1}$ and $x_{2}$, will still hold.  But the second coordinate of $x_{3}$ has been unrestricted by $\Gamma$,
so the cells labeled by $1 < _{1} 2 <_{1} 3$, $1 < _{1} 3 <_{1} 2$ and $3  < _{1} 1 <_{1} 2$ all occur on the boundary of $\Gamma$.  
See \refF{bdcomponents}.
The combinatorics of shuffles will play a central role in the boundary structure in general. 

\begin{center}
\begin{figure}
\begin{tabular}{c}

\xymatrix{
&\includegraphics[width=3cm]{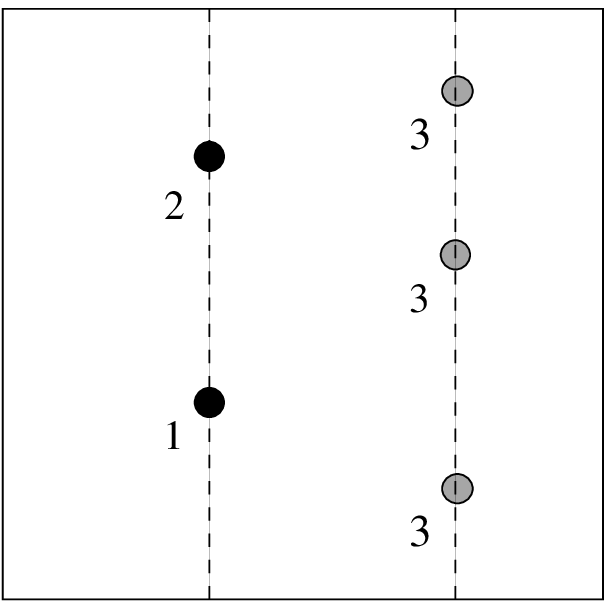}\ar[dl]\ar[d]\ar[dr]&\\
\includegraphics[width=3cm]{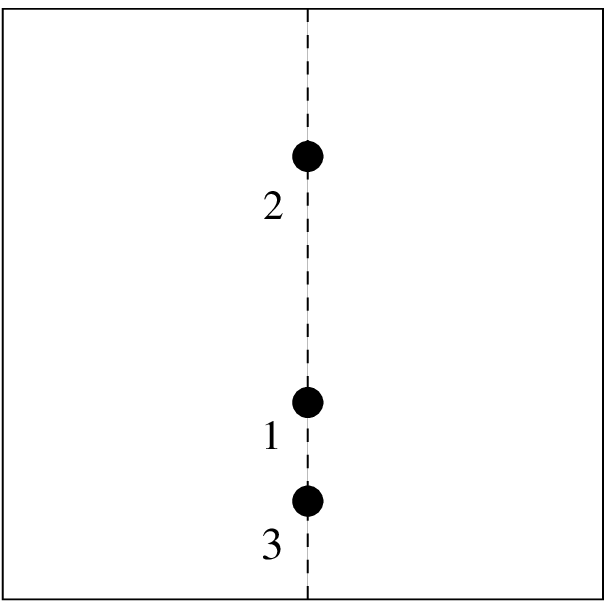}&\includegraphics[width=3cm]{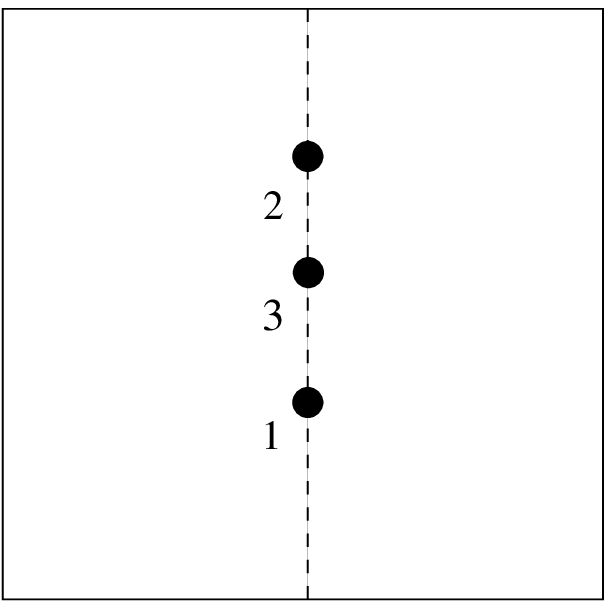}&\includegraphics[width=3cm]{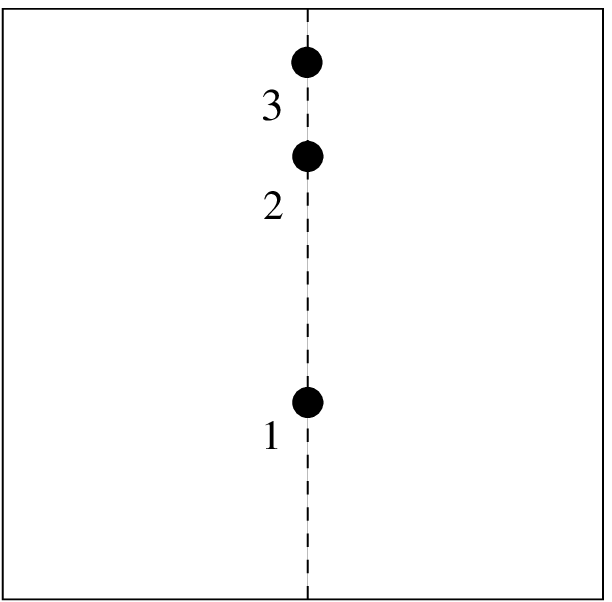}\\
}

\end{tabular}
\caption{The degeneration of the second inequality in the cell $1 <_1 2 <_0 3$ to the bounding cells $3<_1<1<_12$, $1<_1 3<_12$ and $1<_12<_13$ respectively\textmd{} .}

\label{F:bdcomponents}
\end{figure}
\end{center}

 We introduce a  data structure equivalent
to depth orderings in terms of trees in order to more easily describe the boundary combinatorics.   These trees were also used
by Vassiliev \cite{Vass92} to label this cell structure,  and  have been
of interest in higher category theory (see for example the Chapter 7 of the expository book \cite{ChLa04}).
Recall that in a rooted tree, one vertex is under another if it is in the unique path from the root to that vertex.

\begin{defn}
For a depth-ordering $\Gamma = i_{1} <_{a_{1}} i_{2 } < \cdots <_{a_{n-1}} i_{n}$ define a planar level tree $\tau_{\Gamma}$ up to isotopy as follows.
\begin{itemize}
\item There is a root vertex at
the origin, and all others vertices have positive integer-valued heights (that is, $y$-coordinates).
\item There are  $n$ leaves which are of height  $m$ labelled $i_{1}$, $i_{2}, \cdots i_{n}$ in order
from left to right (that is, in order of $x$-coordinates).
\item Under each leaf there is a unique edge from a vertex of height $i$ to one of height $i+1$ for each $0 \leq i < m$.  The leaves
labelled $i_{p}$ and $i_{p+1}$ share the edges under them 
up to height $a_{p}$, but not above this height.
\end{itemize}
\end{defn}

For example, the depth ordering $\Gamma = 3 <_2 5 <_0 1 <_3 2 <_0 4 <_1 6$ is represented by the following tree.
\[
\xymatrix{
3&5&1&2&4&6\\
\bullet\ar[u]&\bullet\ar[u]&\bullet\ar[u]\ar[ru]&&\bullet\ar[u]&\bullet\ar[u]\\
&\bullet\ar[lu]\ar[u]&\bullet\ar[u]&&\bullet\ar[u]&\bullet\ar[u]\\
&\bullet\ar[u]&\bullet\ar[u]&&\bullet\ar[u]\ar[ru]\\
&&\bullet\ar[lu]\ar[u]\ar[rru]\\
}\]

The tree $\tau_{\Gamma}$ uniquely determines $\Gamma$, so we use the two structures interchangeably.  Recall as well that the
planar isotopy class of a level tree is equivalent to an abstract level tree with an ordering (say from left to right) of the edges
incident to each vertex.

\begin{defn}
Let $\tau$ be a planar level tree.   If $e$ and $f$ are two edges  which are incident to the internal vertex $v$, consecutive in
the incident ordering with $e$ before $f$, then the edge quotient $\tau_{/(e=f)}$ is the planar level tree described as an 
abstract level tree with edge orderings as follows.
\begin{itemize}
\item Its edge set is obtained from that of $\tau$ by identifying $e$ and $f$.  We call the resulting special edge the quotient edge.
\item Its vertex set is obtained from that of
$\tau$ by identifying the terminal vertices of $e$ and $f$.  We call the resulting vertex the quotient vertex. 
\item The edges with this terminal vertex as initial are ordered consistently with their previous ordering, with those incident 
to $e$ before those incident before $f$.
\item All other incidence relations and edge orderings are transported by the bijection away from $e$ and $f$
with the edges and vertices of $\tau$.
\end{itemize}
\end{defn}

\begin{defn}
A vertex permutation at $v$ of a planar level tree $\tau$ is a tree which differs from $\tau$ only by changing the edge ordering of 
$v$.  We denote such by $\sigma_{v}  \tau$, where $\sigma_{v}$ is the permutation of edges at $v$.

When $\sigma_{v}$  is a shuffle, we call the resulting tree a vertex shuffle of $\tau$.
When the tree in question is an edge quotient $\tau_{/(e=f)}$, by convention we shuffle at the quotient vertex $v$ using the initial
partition into edges which were incident to $e$ in $\tau$ followed by edges which were incident to $f$. 
\end{defn}

We can now describe the boundary structure of Fox-Neuwirth cells.

\begin{prop}\label{P:fnbdary}
The cell $\Conf{\Gamma'}{m}$ is in the boundary of $\Conf{\Gamma}{m}$ if and only if $\tau_{\Gamma'}$ is isomorphic to
$\sigma_{v} \tau_{\Gamma /(e=f)}$ for some consecutive edges $e$ and $f$ and some shuffle $\sigma_{v}$  at the quotient
vertex.
\end{prop}

For example, let $m=3$ and $\Gamma = 1 <_1 2 <_0 3$.  Then $\Gamma$ is represented by the following tree. 
\[
\xymatrix{
1&2&3\\
\bullet\ar[u]&\bullet\ar[u]&\bullet\ar[u]\\
&\bullet\ar[lu]\ar[u]&\bullet\ar[u]\\
&\bullet\ar[ru]\ar[u]&\\
}
\]

$\Conf{\Gamma}{3}$ has as its boundary precisely cells $\Conf{\Gamma'}{3}$ with $\Gamma'$ represented by each of the following.

\begin{tabular}{ccccc}
\xymatrix{
1&2&3\\
&\bullet\ar[lu]\ar[u]&\bullet\ar[u]\\
&\bullet\ar[u]&\bullet\ar[u]\\
&\bullet\ar[ru]\ar[u]&\\
}
&
\xymatrix{
2&1&3\\
&\bullet\ar[lu]\ar[u]&\bullet\ar[u]\\
&\bullet\ar[u]&\ar[u]\\
&\bullet\ar[ru]\ar[u]&\\
}
&
\xymatrix{
1&2&3\\
\bullet\ar[u]&\bullet\ar[u]&\bullet\ar[u]\\
&\bullet\ar[lu]\ar[u]\ar[ru]\\
&\bullet\ar[u]&\\
}
&
\xymatrix{
1&3&2\\
\bullet\ar[u]&\bullet\ar[u]&\bullet\ar[u]\\
&\bullet\ar[lu]\ar[u]\ar[ru]\\
&\bullet\ar[u]&\\
}
&
\xymatrix{
3&1&2\\
\bullet\ar[u]&\bullet\ar[u]&\bullet\ar[u]\\
&\bullet\ar[lu]\ar[u]\ar[ru]\\
&\bullet\ar[u]&\\
}
\end{tabular}

If we had instead considered $\Conf{\Gamma}{2}$, only the three cells corresponding to the last three in the above list would appear in the boundary since the other potential cells are empty.

There is one edge quotient for each consecutive pair of leaves in the tree $\tau_{\Gamma}$, or equivalently for each inequality 
in $\Gamma$.  For each edge quotient, the number of vertex shuffles varies and is always at least two.

\subsection{The cochain complexes}
We give what is to our knowledge the first explicit formula for the differential in the resulting integral cochain complexes.
In light of \refP{fnbdary}, it remains to determine signs.
Fox-Neuwrith cells are simple to endow with coordinates and thus  
orientations.  Take the standard coordinates of $\Conf{n}{m}$ as a subspace of $\R^{mn}$, namely $(x_{1})_{1}, (x_{1})_{2}, \ldots, (x_{1})_{m},
(x_{2})_{1}, \ldots$, and omit any coordinate which from the definition of $\Conf{\Gamma}{m}$ must be 
equal to one which appears earlier in the list.
Straightforward analysis suffices to understand the difference between the orientation of a cell
and that induced by a cell which bounds it, leading to the following.

\begin{defn}\label{D:shufflesgn}
The sign of the quotient vertex of the $p$th edge quotient of $\Gamma=i_1<_{a_1}i_2<\cdots<_{a_{n-1}}i_n$, denoted $\sgn_{m} (p)$, $p\in\{1,\dots,n-1\}$ is defined as $\sgn_{m}(p, \Gamma) = (-1)^\kappa,$ where
\begin{equation*}
 \kappa = {p+\sum_{k=1}^{p-1} \left(m - \text{min}\{a_k, a_p+1\}\right) + (m - (a_p+1)) + \sum_{k=p+1}^{n-1} \left(m - \text{min}\{a_k, a_p\}\right)}.
\end{equation*}
\end{defn}

If we consider only configurations in $\mathbb{R}^m$ with $m$ even, we can delete $m$ from the definition of $\kappa$. 
We now only concern ourselves with the limit as $m$ goes to infinity.

\begin{defn}
Define the cochain complex ${ {\widetilde{\fn_n}}}^{*}$ as follows.  
As a free abelian group it is generated by depth-orderings, with  $[\Gamma]$ 
in degree $\sum a_{p}$ when $\Gamma = i_{1} <_{a_{1}} i_{2} < \cdots <_{a_{n-1}} i_{n}$.  The differential is
$$d[\Gamma] = \sum_{p} \sgn_{0}(p, \Gamma)a_{p} [\Gamma], \;\; {\text{where}} \;\;  a_{p} [\Gamma] = \sum_{\sigma_{v}} [\Gamma'].$$
The first sum is over inequalities in $\Gamma$, indexed by $p$, which in turn determine an adjacent edge pair $e,f$ 
in $\tau_{\Gamma}$.
For each such edge pair we sum
over all possible vertex shuffles of the quotient tree
at the quotient vertex, taking the depth-ordering $\Gamma'$ associated
to the resulting tree $\sigma_{v} \tau_{\Gamma /(e=f)}$.
\end{defn}

\begin{thm} \label{T:fnmodel}
${ {\widetilde{\fn_n}}}^{*}$ is a cochain model for $E \si_{n}$, so its quotient by $\si_{n}$ which we call ${ {{\fn_n}}}^{*}$ has cohomology isomorphic to that of  the $n$th symmetric group.
\end{thm}

In algebraic terms, ${ {\widetilde{\fn_n}}}^{*}$ is a free $\Z[\si_{n}]$ resolution of the trivial module.  It is a challenge to even show
algebraically that ${ {\widetilde{\fn_n}}}^{*}$ is a cochain complex, much less that it is acyclic.

\begin{proof}[Sketch of proof]
By \refT{cells}, the $[\Gamma]$ naturally span the corresponding 
cellular chain complex for the one-point compactification $\Conf{n}{m}^{+}$, with $[\Gamma]$
in $C^{CW}_{mn - \sum a_{i}}\left(\Conf{n}{m}^{+}\right)$.  By Alexander and Spanier-Whitead duality for manifolds as explained by 
Atiyah \cite{Atiy61}, 
the resulting homology group is isomorphic to cohomology of $\Conf{n}{m}$ in degree $k = \sum a_{i}$.  These 
cochain groups are isomorphic through the maps induced by the directed system in degrees $k$ less than $m$, so the inverse
limit is just $ {\widetilde{\fn_n}}^{k}$.  

The main work in proof is then to show that the boundary homomorphisms 
have the indicated signs. 
\end{proof}

We now consider more closely the quotient complex ${ {\fn_n}}^{*}$ whose cohomology is that of $B \si_{n}$.  
The equivalence class of some $\Gamma = i_{1} <_{a_{1}} < i_{2}< \cdots <_{a_{n-1}} i_{n}$ modulo $\si_{n}$, which 
acts by permuting the $i_{k}$, is given by the subscripts of the inequalities which we now denote
$\mathbf{a} = [a_{1}, \ldots, a_{n-1}]$.  These correspond to trees $\tau_{\mathbf{a}}$ as described above but with leaves now unlabelled. Thus ${ {\fn_n}}^{*}$ is significantly smaller
than the bar complex, having rank in degree $k$ of the number of partitions of $k$ into  $n-1$ non-negative integers.  
We express the differential of ${ {\fn_n}}^{*}$ in terms of $\mathbf{a}$.

\begin{defn} 
An $\ell$-block of $\mathbf{a}$ is a maximal (possibly empty) sub-sequence of consecutive $a_i$ greater 
than $\ell$ in $\mathbf{a}$. Denote the ordered collection of all $\ell$-blocks of $\mathbf{a}$ by $B_\ell(\mathbf{a})$.
\end{defn}

For example, $B_0([0,2,6,0,1]) = (\emptyset, [2,6], \emptyset, [1])$ and its $B_1$ is $(\emptyset, [2,6], \emptyset, \emptyset)$. These collections $B_\ell(\mathbf{a})$  correspond to the forests of rooted trees obtained by deleting from $\tau_{\mathbf{a}}$ all vertices of height lower than $\ell$ and their incident edges. 

\begin{defn}
Let $\mathbf{a} \in { {\fn_n}}^{k}$. Denote by $\mathbf{a}[i]$ the element 
$[a_1,a_2,\dots,a_{i-1}, a_i+1, a_{i+1},\dots,a_{n-1}] \in { {\fn_n}}^{k+1}$.

Let $a_i \in \mathbf{a}$ and denote by $\mathbf{A}_i$ the $a_i$-block of $\mathbf{a}[i]$ containing $a_i$. Such a block $\mathbf{A}_i$ corresponds to a rooted subtree $\tau_{\mathbf{A}_i}$ of $\tau_{\textbf{a}[i]}$.
Let $N = \# B_{a_i+1}(\mathbf{A}_i)$, the number of trees in the forest obtained by removing the root from $\tau_{\mathbf{A}_i}$ along with its incident edges, and let $k$ to be the number of trees in this forest whose roots were incident to the edge $e$ in the tree $\tau_{\textbf{a}}$ before the edge quotient $e=f$ which produced $\tau_{\textbf{a}[i]}$. Define $\Sh(\mathbf{a},i)$ to be the set of $(k, N-k)$-shuffles.  
\end{defn}

$\Sh(\mathbf{a}, i)$ acts on $\mathbf{a}[i]$ by shuffling the elements of $B_{a_i+1}(\mathbf{A_i})$. Equivalently, it acts via vertex shuffles on the tree $\tau_{\mathbf{a}[i]}$. 


\begin{prop} \label{P:coboundarymap} The differential in ${ {\fn_n}}^{*}$ is given by
\begin{equation*}
\delta(\mathbf{a}) = \sum_{i=1}^{n-1} (-1)^{i + \alpha(i)} \delta_{i} (\mathbf{a}) \;\;
\text{where} \;\;  \delta_{i} (\mathbf{a})=  \sum_{\sigma \in \Sh(\mathbf{a},i)} (-1)^{\kappa(\sigma, \mathbf{a})} \sigma\cdot\mathbf{a}[i].
\end{equation*}
If we define $h(\mathbf{v})$ to be the height in the tree $\tau_{\mathbf{a}}$ of the vertex $\mathbf{v}$ and $\#(\sigma, h)$ to be the total number of transpositions of vertices at height $h$ in $\tau_{\mathbf{A}_i}$ which occur when $\sigma$ acts on the tree, then we have

\begin{equation*}
\alpha(i) = \sum_{j=1}^{i-1} \text{min}\{a_j, a_i+1\} + a_i+1 + \sum_{j=i+1}^{n-1}\text{min}\{a_j, a_i\},
\;\; \text{and} \;\;  
\kappa(\sigma, \mathbf{a}) = \sum_{\mathbf{v}} (\#(\sigma, h) - \#(\sigma, h-1))\cdot h(\mathbf{v}),
\end{equation*}
where the second sum is indexed over vertices $\mathbf{v}$ in the subtree $\tau_{\mathbf{A}_i}$.
\end{prop}

The signs and the differential as a whole do not seem to simplify, as evidenced by the following examples of computations,
which spite of the superficial similarity in cell names have substantially different boundaries.

\begin{eqnarray*}
\delta([2,0,1,2]) &=& - 3 [2,0,2,2] - 1 [2,1,1,2] \\
\delta([1,0,2,2]) &=& - 1 [1,1,2,2] + 1 [1,2,2,1] + 2 [2,0,2,2] - 1 [2,2,1,1]\\
\delta([0,2,1,2]) &=& - 6 [0,2,2,2] - 1 [1,2,1,2] + 1 [2,1,1,2] - 1 [2,1,2,1]\\
\delta([0,1,2,2]) &=& - 4 [0,2,2,2] - 1 [1,2,2,1].
\end{eqnarray*}

The first author has coded a program in Java to compute these differentials.  It is appended at the end of the .tex file 
which was uploaded for the arXiv preprint of this paper, and is also available on his web site.

\section{The cohomology of $B\si_{4}$}


In this section we establish the following.

\begin{thm} \label{T:bs4basis} There is an additive basis for $H^*(B\si_4;\F_2)$ with representatives given by elements of
${\fn_{4}}^{*}$ of the following forms
\begin{itemize}
\item $[a,0,0] + [0,a,0] + [0,0,a]$, with $a > 0$
\item  $[a, 0, b] + [b,0,a]$, with $0 < a < b$ 
\item  $[b, a, b]$, with $0 \leq a \leq b$.
\end{itemize}
\end{thm}

These Fox-Neuwirth representative cocycles along with the general representatives we find below were also found by Vassiliev 
\cite{Vass92}, though he did not use them to determine multiplicative structure as we do and only sketched 
the proof that these cocycles form a basis.   
Because mod-two homology of symmetric groups is well-understood
\cite{Naka61, CLM76}, we based our calculations in \cite{GSS09} on knowledge of homology, as did Vassiliev.  
We take the opportunity to give a complete proof of this result
for $B \si_{4}$ to show that a self-contained treatment of cohomology is feasible.

Throughout this section, we 
label the entries of a basis element of ${\fn_{4}}^{*}$ by $a,b,c$ with $a \leq b \leq c$.


\begin{proof}
That the cochains listed are cocycles is a straightforward calculation using the formula for the differential $\delta$ in the 
Fox-Neuwirth cochain complex given in \refP{coboundarymap}.

Cocycles which are ``disjoint from'' (that is, project to zero in) the subspaces on which our generating cocycles are naturally defined are null-homologous, 
as we will establish through the construction of a chain homotopy operator. 
Consider the submodule $S$
of ${\fn_{4}}^{*}$
spanned by cochains whose entries are all positive, and not of the form $[b,a,b]$, $[a,a,b]$ or $[b,a,a]$ with $a \leq b$.  
We define a chain operator  $P : S \to {\fn_{4}}^{*}$ of degree $-1$ which
lowers the smallest entry by one or maps to zero, depending on the order of the blocks with respect to the minimum entry.  
Explicitly, we have the following, where $0< a< b< c$.

\begin{itemize}
\item  $[a, b, c] \overset{P}{\mapsto} [a-1, b, c],$  and $[a, c, b] \overset{P}{\mapsto} [a-1, c, b],$ and $[a,b,b] \overset{P}{\mapsto} [a-1, b, b]$
\item $[b, c, a]$, $[c,b,a]$ and $[b,b,a]$ all map to zero. 
\item  $[b, a, c] \overset{P}{\mapsto} [b, a-1, c],$
\item  $[c, a, b]$ maps to zero.
\end{itemize}

Let $\chi$ be the automorphism of $S$ which exchanges $a$-blocks, which is a chain map.  
That is $\chi [a,b,c] = [b,c,a]$, while $\chi [b,a,c] = [c,a,b]$, and so forth.
We claim that on $S$ the map $P$ is a chain null-homotopy of $\chi$.  For example if $b > a+1$ and $c > b+1$
then 
$$P \delta [a,b,c] = P\left( [a+1, b, c] + [b,c, a+1] + [a, b+1, c] + [a,c,b+1]\right) = [a,b,c] + 0 + [a-1,b+1,c] + [a-1,c,b+1],$$
while
$$\delta P [a,b,c] = \delta [a-1, b,c] = [a,b,c] + [b,c,a] + [a-1, b+1, c] + [a-1, c, b+1],$$
so their sum is $[b,c,a]$.  Removing the restrictions $b > a+1$ and $c > b+1$ 
and treating other cases all follows from similar direct computation.

Thus for any cocycle $\gamma$ in $S$, 
$$\delta (P \chi (\gamma)) = P (\delta \chi (\gamma)) + \chi (\chi (\gamma)) = \gamma,$$
so any cocycle in $S$ is trivial in cohomology.

We now show that a cycle is always homologous to one whose image in subspaces complementary to $S$ is given by the cycles
we have named, which seems to require ad-hoc analysis.

Consider chains whose first entry is zero, but whose second and third are not, which we call $S_{f0}$.  
When we compose the differential with projection
onto $S_{f0}$, we obtain a complex which we claim is acyclic, as can be seen using a ``decrease the smallest nonzero
entry'' nullhomtopy
as above.  Thus any chain which projects to $S_{f0}$ non-trivially is homologous to one which projects to it trivially.  
Similarly, we can rule
out cycles involving chains whose last entry is zero but first two are not, namely $S_{l0}$.

For  chains whose middle entry alone is zero $S_{m0}$, we have $\delta[a,0,b] = [a,1,b] + [b,1,a]$.  Cochains with a middle entry of one 
do not otherwise appear in the image of $\delta$, so the projection of a cycle onto $S_{m0}$ must itself be a cycle, 
which direct calculation shows must be of the form $[a,0,b] + [b,0,a]$.

Similarly,  if 
we consider chains with only one-nonzero entry and two zeros, which we call $S_{00}$, the boundary involves chains with a single zero and a single one.
Only chains with two zeros can have such boundaries, so the projection of a cycle onto $S_{00}$ must itself be a cycle, which
direct calculation shows must be of the form $[a,0,0] + [0,a,0] + [0,0,a]$.

Finally, consider $S_{2a}$, which is the span of $[a,a,b]$, $[a,b,a]$, and   $ [a,a,b]$.  Assume at first that
 $a + 1< b$, in which case the image of $[a,a,b]$ under $\delta$ is $[a,a+1, b] + [a, b, a+1]$.  The only other cochain for which either
of these appears non-trivially in its coboundary is $[a,b,a]$, whose coboundary consists of those along with $[b, a+1, a]$ and $[a+1, b, a]$.
These latter two terms are the coboundary of $[b,a,a]$, so the only cycle in which these occur nontrivially 
is $[a,a,b] + [a,b,a] + [a,a,b]$.  But this is trivial in cohomology, equal to $\delta ([a, a-1,b])$.  Thus any cycle is homologous
to one which projects trivially onto  $S_{2a}$.  The analysis with $a+1 = b$ is similar.  

Because ${\fn_{4}}^{*} $ is spanned by $S$, $S_{f0}$, $S_{l0}$, $S_{m0}$, $S_{00}$, $S_{2a}$,
and the cycles listed in the statement of the theorem, the result follows.
\end{proof}

\subsection{Block symmetry and skyline diagrams}
The cochain representatives for the cohomology of $B\si_{4}$ all exhibit some symmetry.  Define a block permutation to be a permutation
of the entries of a Fox-Neuwirth basis cochain which does not change its (unordered) set of $i$-blocks for any $i$.   The Fox-Neuwirth cochains which comprise each of the 
cocyles listed in \refT{bs4basis} are invariant under block-permutations.  In particular, if we let $\Symm({\bf a})$ denote the sum of all
distinct block-permutations of ${\bf a}$, then the first two sets of generators are $\Symm([a,0,b])$  for $0 \leq a < b$  and the 
third is $\Symm([b,a,b])$,  for $0 \leq a \leq b$.
Symmetric representatives exist for  the cohomology $B \si_{n}$ more generally, as we will see below.  
It might be enlightening to have a more direct 
proof of this block-symmetry property than through simply observing it holds after full computation.

The symmetry of these cochains allows the essential data defining them to be given in a graphical representation which we call
\emph{skyline diagrams}. The basis elements described in \refT{bs4basis} correspond to skyline diagrams in \refF{bs4skyline}.
These will be treated fully in Section~\ref{S:skyline}.

\begin{figure}
\subfigure[$\Symm({[}a,0,b{]}), 0 \leq a < b$]{
\psfrag{a1}{$b$}
\psfrag{a2}{$a$}
\includegraphics{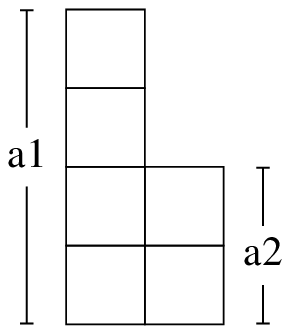}
}
\subfigure[${[}b,a,b{]}$, $0 \leq a \leq b $]{
\psfrag{b1mb2}{$b -a$}
\psfrag{b2}{$a$}
\includegraphics{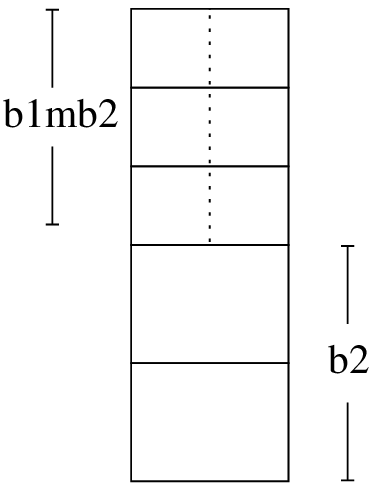}
}
\caption{Skyline diagrams for the basis of $H^*(B\si_4; \F_2)$ described in \refT{bs4basis}}
\label{F:bs4skyline}
\end{figure}

\subsection{Geometry and characteristic classes}
We can understand our cocycles geometrically through intersection theory, which is one of the standard ways to understand
duality.  If $M$ is a manifold without boundary and $M^{+}$ is its one-point compactification, then at the cochain level 
Alexander-Spanier-Whitehead 
duality identifies a cellular codimension-$d$ cell $C$ in $M^{+}$ through a zig-zag of maps with a cochain on $M$ 
whose value on
chains $\sigma : \Delta^{d} \to M$ which are transversal to the interior of $C$  is the mod-two count of $\sigma^{-1}(C)$.  In our setting,
these counts are of configurations whose points share prescribed coordinates.

For example, in $H_{4}(B\si_{4})$ there is the fundamental class of a submanifold $\R P^{2} \times \R P^{2}$ of $\Conf{4}{\infty}$.
This submanifold parameterizes configurations of four points two of which are on the unit sphere say centered at the origin and two of which 
are on a unit sphere chosen to be disjoint from that one.  The cocycle $[2,0,2]$ evaluates non-trivially on this homology class,
as can be seen by a count of configurations in this submanifold and in the cell, as we picture in \refF{pairing}.  

The corresponding 
characteristic class can be evaluated on a four-sheeted covering space by embedding it in a trivial Euclidean bundle and taking the
Thom class of the locus where the fibers of the cover can be partitioned into two groups of two, each of which share their first two coordinates.
For the cocyle $[2,1,2]$, the corresponding characteristic class would be the Thom class of the locus where all four points in the fiber
share their first coordinate, and they can be partitioned into two groups of two which share an additional coordinate.

\begin{figure}
\xymatrix{
\includegraphics[width=4cm]{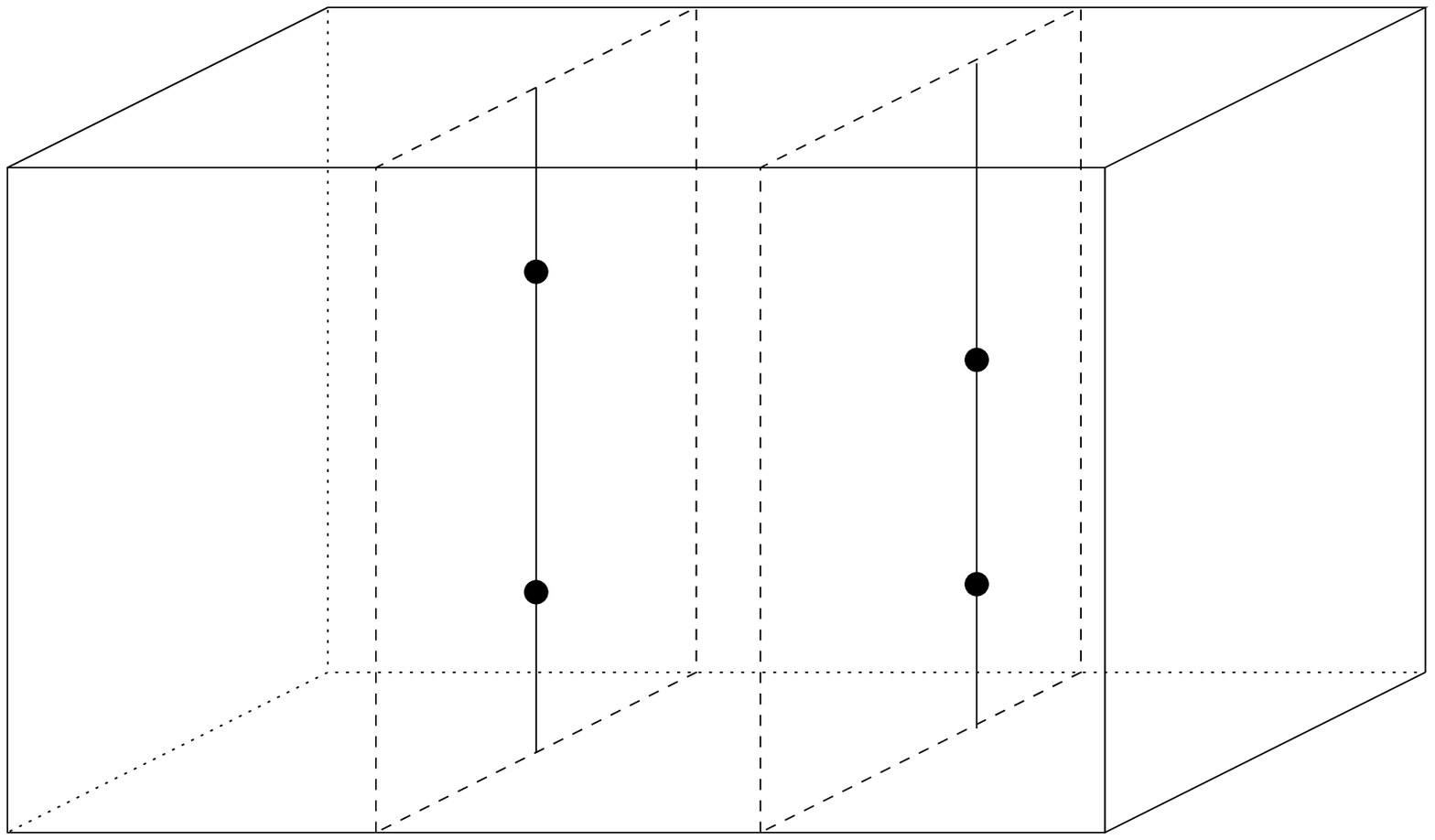}\ar[dr]&&\includegraphics[width=4cm]{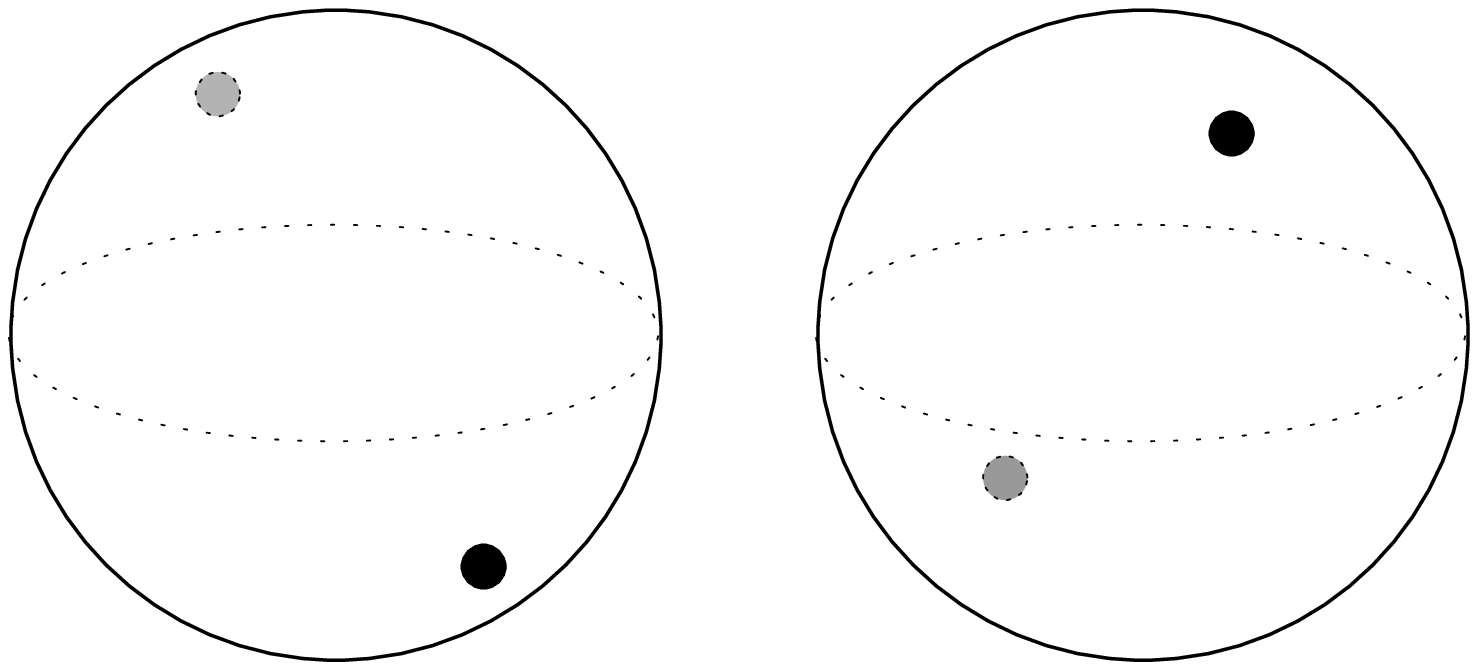}\ar[dl]\\
&\includegraphics[width=4cm]{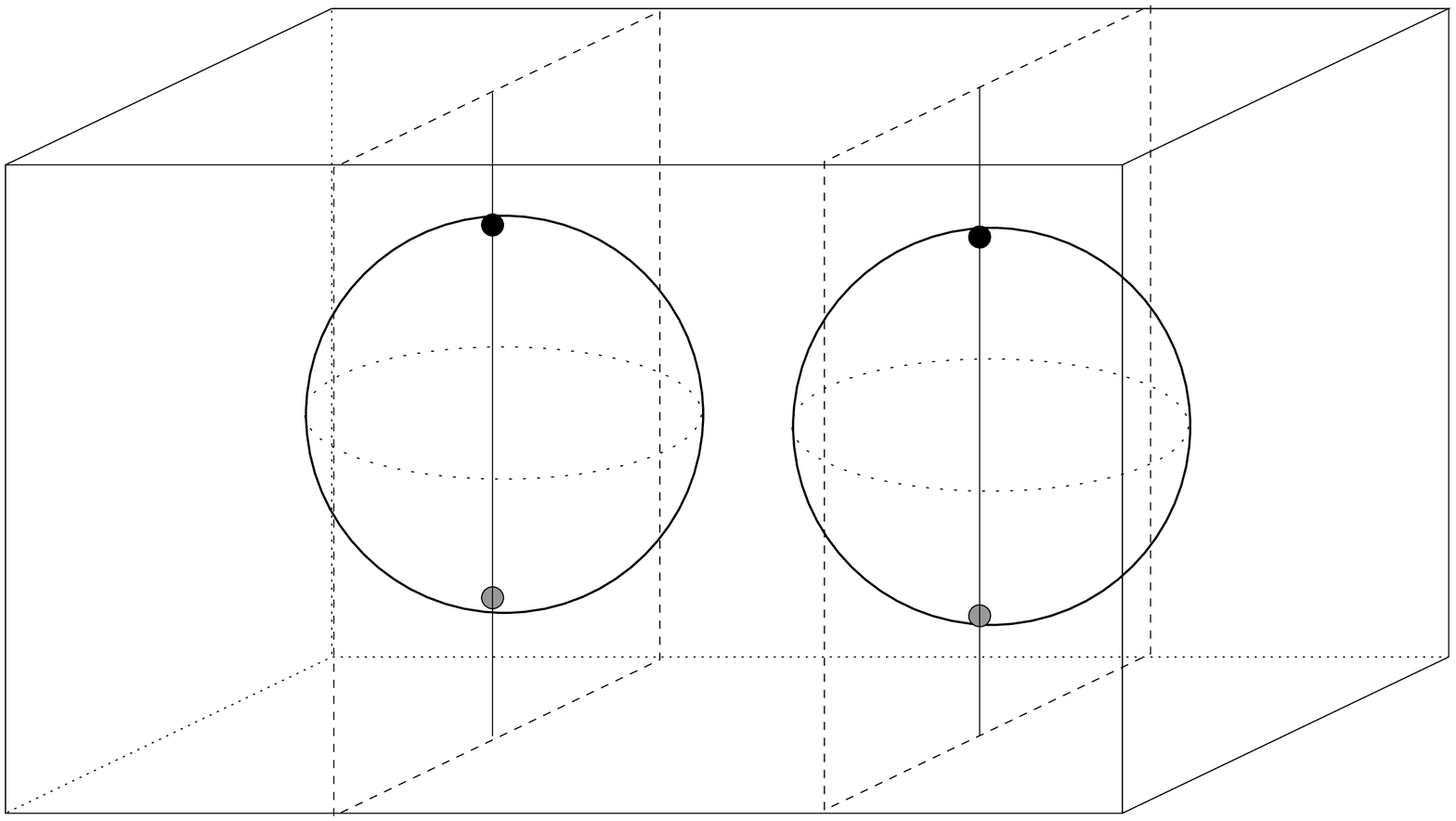}&
}
\caption{The cocycle $[2,0,2] \in H^4(B\si_4)$ pairs non-trivially with the fundamental class of $\R P^2\times \R P^2$ considered as a submanifold of $\Conf{4}{\infty}$ as indicated.}
\label{F:pairing}
\end{figure}

\section{Hopf ring structure and presentation}

While cup product structure typically clarifies the calculation of cohomology, the cohomology rings of symmetric
groups have been notoriously difficult to understand.  Feshbach's description \cite{Fesh02} is complicated, with recursively presented
relations.   We have found that the situation is clarified once one uses 
a second product structure on cohomology, along with a coproduct (and antipode which is trivial)
in order to obtain what is known as a Hopf ring.  This structure was first discovered by Strickland and Turner \cite{StTu97}.

\subsection{Cup product structure}
\begin{defn}\label{D:intprod}
Let $\mathbf{a}=[a_1,\dots,a_{n-1}]$ and $\mathbf{b}=[b_1,\dots,b_{n-1}]$ be basis elements in the mod-two Fox-Neuwirth
cochain complex. 
Define their intersection product $\mathbf{a}\cdot\mathbf{b}$ to be
\begin{equation*}
[a_1,\dots,a_{n-1}] \cdot [b_1,\dots,b_{n-1}] = [a_1+b_1,\dots,a_{n-1}+b_{n-1}],
\end{equation*}
\end{defn}

Integrally, we would also have the usual sign accounting for the possible difference in orientations.  

\begin{lem} 
The intersection product makes $\fn^*$ a (commutative) differential graded algebra.
\end{lem}

\begin{prop}
Through the isomorphism of \refT{fnmodel}, the product on cohomology induced by the intersection product  on ${\fn_{n}^{*}}$ agrees
with the  cup product on cohomology of $B\si_{n}$.  
 \end{prop}
 
 We call this the intersection product in light of the fact mentioned above that  the cohomology classes which arise from 
 the Fox-Neuwirth cell structure on $\UConf{n}{D}^{+}$ are Thom classes of the union of corresponding cells in  $\UConf{n}{D}$.
 These unions of cells correspond to images of manifolds whose fundamental classes in locally finite homology are dual to the Thom classes,
 as in Definition~4.6 of \cite{GSS09},
 and thus are appropriate for elementary intersection theory.
We can generalize the Fox-Neuwirth cell structure to a family of such in which specified sets of coordinates other than 
the first $a_{i}$ coordinates must be equal, with the resulting cohomology classes being independent of which are chosen.
We then use that the cup product of Thom classes of two varieties is the Thom class of their intersection (when transversal).
For example the set of configurations in which the first two points have
their first $a_{1}$ coordinates agree intersected with
the set in which the the first two points have the  next $b_{1}$ coordinates equal will of course be the set in which the
first two points have their first $a_{1} + b_{1}$ equal, and so forth, showing that the intersection product in the Fox-Neuwirth cell structure
corresponds to the product structure of the associated Thom classes.

Using the intersection product, it is straightforward to compute the cup product structure appearing in \refF{cupprod} on $H^*(B\si_4;\F_2)$,
using the skyline basis for convenience. 
The computations use observations that were made in the proof of \refT{bs4basis} of cocycles which are coboundaries.
Using the graphical skyline presentation, multiplication consists simply of ``stacking columns in all possible ways'', 
with a result of zero if a vertical line does not continue for the entire height of the column.  See the end of Section~\ref{S:skyline} for 
the general recipe, and Section~6 of \cite{GSS09} for a detailed account.

\begin{figure}
\begin{eqnarray*}
\psfrag{a1}{$a_1$}\psfrag{a2}{$a_2$}\psfrag{a3}{$a_3$}\psfrag{a4}{$a_4$}\raisebox{-1.4cm}{\includegraphics{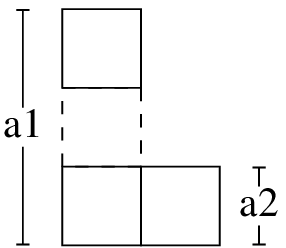}} \ccup \raisebox{-1.4cm}{\includegraphics{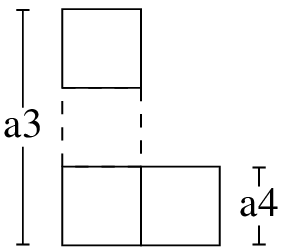}}&=&\psfrag{a2a4}{$a_2+a_4$}\psfrag{a1a3}{$a_1+a_3$}\overline{\delta}_{1+3,2+4}\raisebox{-1.4cm}{\includegraphics{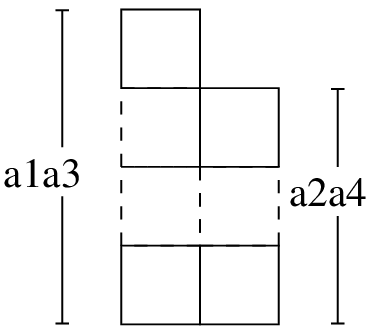}} \\
&&+ \psfrag{a2a3}{$a_2+a_3$}\psfrag{a1a4}{$a_1+a_4$}\overline{\delta}_{1+4,2+3}\raisebox{-1.4cm}{\includegraphics{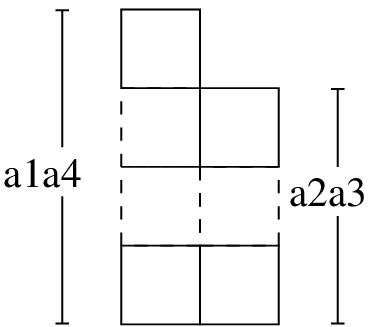}}\\\\\psfrag{a1}{$a_1$}\psfrag{a2}{$a_2$}\raisebox{-1.4cm}{\includegraphics{images/Skyline_BS4_1_a}} \ccup \psfrag{b1mb2}{$b_1$}\psfrag{b2}{$b_2$}\raisebox{-1.4cm}{\includegraphics{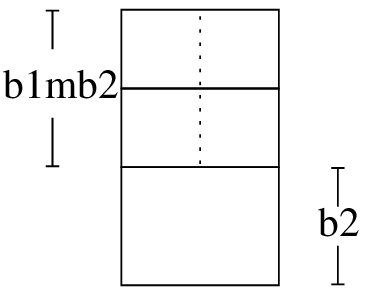}}& =& \begin{cases}
\psfrag{a1}{$a_1$}\psfrag{a2}{$a_2$}\psfrag{b1}{$b_1$}\raisebox{-1.4cm}{\includegraphics{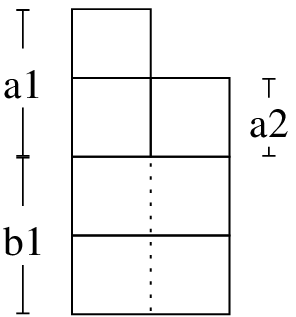}} 
&\text{if } b_2 = 0\\\\
0  &\text{if } b_2 \neq 0
\end{cases}\\\\\psfrag{b1mb2}{$b_1$}\psfrag{b2}{$b_2$}\raisebox{-1.4cm}{\includegraphics{images/Skyline_BS4_6}}  \ccup \psfrag{b3mb4}{$b_3$}\psfrag{b4}{$b_4$}\raisebox{-1.4cm}{\includegraphics{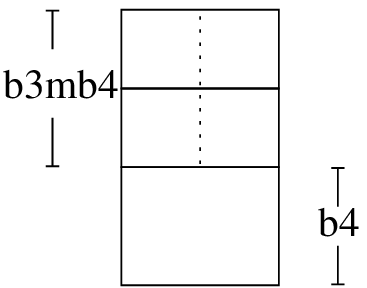}}& =& \psfrag{b1pb3}{$b_1+b_3$}\psfrag{mb2mb4}{}\psfrag{b2pb4}{$b_2+b_4$}\raisebox{-1.4cm}{\includegraphics{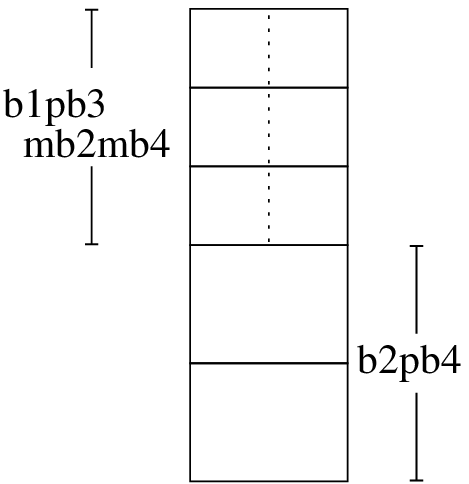}} 
\end{eqnarray*}
\caption{The cup product structure on $H(B\si_4)$. Here,  $\overline{\delta}_{i+j,k+\ell}$ is $1 + \delta(a_{i}+a_{j}, a_{k}+ a_{\ell})\in \F_2$ 
where $\delta$ is the usual Kronecker delta function.}
\label{F:cupprod}
\end{figure}

For a ring presentation, denote cohomology classes using only cochains, and let 
$\x = \Symm[1,0,0]$, $\y = [1, 0, 1]$ and $\z = [1,1,1]$.  Then $H^{*}(BS_{4}; \F_{2}) \cong \F_{2}[\x,\y,\z]/ (\x \z)$.  This presentation
is deceptively simple.  At present there are only recursive presentations of the relations in the cohomology rings of general symmetric groups.

\subsection{Hopf ring structure}
We have found that the best way to organize the relations in the cohomology rings of symmetric groups, which seem inherently recursive, is through a Hopf ring structure

\begin{defn}\label{D:Hopfring}
A Hopf ring is  a ring object in the category of coalgebras.  Explicitly, a Hopf ring is
vector space $V$ with two multiplications, one 
comultiplication, and an antipode $(\tr, \cdot, \Delta, S)$ such that the first multiplication forms a Hopf algebra
with the comultiplication and antipode, the second multiplication forms a bialgebra with the comultiplication,
and these structures satisfy the distributivity relation 
\begin{equation*}
\alpha \cdot (\beta \tr \gamma) =  \sum_{\Delta \alpha = \sum a' \otimes a''} (a' \cdot \beta) \tr 
(a'' \cdot \gamma).
\end{equation*}
\end{defn}

Hopf rings were introduced by Milgram \cite{Milg70}, and arise in topology as one structure governing the homology of 
infinite loop spaces which represent ring spectra.  We give examples arising from algebra in Section~2 of \cite{GSS09} and will give some
explicit calculations below.

\begin{defn}
Let $\mathbf{a}=[a_1,\dots,a_{n-1}]$, and  by convention set $a_{0} = a_{n} = 0$. 
The coproduct  of $\eh \in {\fn_{n}}^{*}$,  is the sum $\Delta \eh = \sum_{a_{i} = 0} [a_{1}, \ldots, a_{i-1}] \otimes [a_{i+1}, \ldots, a_{n}]$
in $\bigoplus_{i+j=n} {\fn_{i}}^{*} \otimes {\fn_{j}}^{*}$.  
\end{defn}

Thus if $a_{i} > 0$ for $1 \leq i \leq n-1$ then $\mathbf{a}$ will be primitive.

\begin{defn}
The transfer product of $\eh \in {\fn_{i}}^{*}$ and $\bb \in {\fn_{j}}^{*}$, denoted $\eh \tr \bb$, is the sum of cochains in ${\fn_{i+j}}^{*}$ whose zero blocks are shuffles of the zero blocks of $\eh$ and $\bb$.
\end{defn}

\begin{thm}
The coproduct and transfer product induce a well-defined coproduct and product respectively on cohomology of $\bigoplus_{n}{\fn_{n}}^{*}.$

Through the isomorphism of \refT{fnmodel}, the induced coproduct and product on cohomology  agree
with the direct sum over $i+j = n$ 
of natural maps and transfers respectively associated to the standard inclusions of symmetric groups
$\si_{i} \times \si_{j} \hookrightarrow \si_{n}$.  
\end{thm}

Thus, the coproduct and transfer product  on cohomology agree with the standard coproduct and the transfer product of Strickland and Turner \cite{StTu97}. They show that the cohomology of the disjoint union of the symmetric groups is a Hopf ring with these along with the 
cup product, so 
we have the following corollary, once we define 
the intersection product between cochains in different summands of  $\bigoplus_{n}{\fn_{n}}^{*}$
to be zero.  

\begin{cor}
Through the isomorphism of \refT{fnmodel},  the intersection product,  
transfer product, and coproduct induce a Hopf semiring 
structure on the cohomology of  $\bigoplus_{n}{\fn_{n}}^{*}.$  With mod two coefficients, the identity map
defines an antipode which with these structures yields a Hopf ring.
\end{cor}

An algebraic proof of this corollary is also possible.
At the chain level, $(\cdot, \Delta)$ and $(\tr, \Delta)$ both induce bialgebra structures on $\bigoplus_{n}{\fn_n}^*$, but the Hopf ring distributivity relation fails. However, there is a Hopf semiring structure on the subcomplex of symmetrized chains.  
This is  analogous to rings of symmetric invariants, which are treated in Section~2 of \cite{GSS09}.
For any algebra $A$ which is flat over its ground ring $R$, the total symmetric invariants, $\bigoplus_n (A^{\otimes n})^{\si_n}$, forms a Hopf semiring.   The bialgebra structures hold before taking invariants, but Hopf ring distributivity requires taking invariants.

We can use Hopf ring distributivity of the cup product over the transfer product  to understand the cup product.  Recall
that  $\Symm(\bf{a})$
denote the sum of all cochains related to  $\bf{a}$ by a block-permutation.
To compute $\Symm([3,0,2,2,2]) \cdot [\Symm([4,0,2,0,2])]$, we can decompose $\Symm[4,0,2,0,2]$ as $[4] \tr [2,0,2]$. 
Now, $\Delta(\Symm([3,0,2,2,2])) = [3] \otimes [2,2,2]$ plus terms which must cup to zero with the chains in our decomposition of $\Symm[4,0,2,0,2])$ because they lie in the cohomology of other symmetric groups. Thus we have

\begin{eqnarray*}\label{productcomp}
\Symm([3,0,2,2,2]) \cdot \Symm([4,0,2,0,2]) &=& \left(([3] \cdot [4]) \tr ([2,2,2] \cdot [2,0,2])\right)\\
&=&[7] \tr [4,2,4]\\
&=&[7,0,4,2,4] + [4,2,4,0,7]\\
&=& \Symm([7,0,4,2,4]).
\end{eqnarray*}

\subsection{Hopf ring presentation}
We now develop the most basic classes, whose cup and transfer products will yield all mod-two cohomology of symmetric groups.

\begin{defn}
Let $\mathbf{g}_{\ell,n}$ be the basis element of $(\fn_{n2^{\ell}})^{n(2^{\ell} - 1)}$ with $n$ 0-blocks, each of which consists of $(2^\ell-1)$ consecutive entries of 1.   Let  $\gamma_{\ell, n} = [\mathbf{g}_{\ell, n}]$. 
\end{defn}

For example, $\mathbf{g}_{2,3} = [1,1,1,0,1,1,1,0,1,1,1] \in C^{9}(B\si_{12})$.  The following is an elaboration of the
main result our paper with Salvatore \cite{GSS09}, with Fox-Neuwirth cochain representatives now given.

\begin{thm}\label{T:genandrel}
As a Hopf ring, $H^{*}(\coprod_{n} B \si_{n}; \F_{2})$ is generated by  the 
classes $\gamma_{\ell, n} \in H^{n(2^{\ell} - 1)}(B \si_{n 2^{\ell}})$,
along with unit classes on each component.
The coproduct of $\gamma_{\ell, n}$ is given by
$$\Delta \gamma_{\ell, n} = \sum_{i+j = n} {\gamma_{\ell, i}} \otimes {\gamma_{\ell, j}}.$$
Relations between transfer products of these generators are given by 
$$\gamma_{\ell, n} \tr \gamma_{\ell, m} = \binom{n+m}{n} \gamma_{\ell, n+m}.$$
The antipode is the identity map.
Cup products of generators
on different components are zero, and there are no other relations between cup products of generators.
\end{thm}

For example, the cocycle $\Symm([4,3,4,1,4,3,4,0,3,3,3,0,1,0,1,0,1,0,0])$ 
is represented as the Hopf ring monomial 
$\gamma_{3,1}{\gamma_{2,2}}^2\gamma_{1,4}\tr {\gamma_{2,1}}^3\tr\gamma_{1,2}\tr\gamma_{1,1}\tr 1_2$ 
and, as we will see in the next section, the skyline diagram
\begin{center}
\includegraphics[width=4cm]{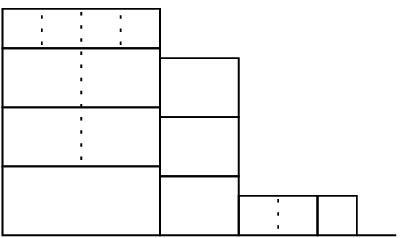}.
\end{center}

An immediate corollary of 
this theorem is that all cohomology classes have block-symmetric representatives, which was  
also found by Vassiliev \cite{Vass92}.

\section{The skyline basis}\label{S:skyline}

A Hopf ring monomial in classes $x_{i}$ is one of the form $f_{1} \tr f_{2} \tr \cdots \tr f_{k}$, where
each $f_{j}$ is a monomial under the $\cdot$ product in the $x_{i}$.
 Because of Hopf ring distributivity,  it is convenient to use Hopf ring monomials in chosen generators to span a Hopf ring.  
In this section, we give a graphical presentation of a Hopf ring monomial basis for the mod-two
 cohomology of symmetric groups, using both spatial dimensions to represent the two products.

\begin{defn}
A gathered monomial in the cohomology of symmetric groups is a Hopf ring monomial  in the generators $\gamma_{\ell, n}$  where such $n$ are maximal or equivalently the number of transfer products which appear is minimal.
\end{defn}

For example, $\gamma_{1,4} {\gamma_{2,2}}^{3} \tr \gamma_{1,2} {\gamma_{2,1}}^{3}
= \gamma_{1,6}  {\gamma_{2,3}}^{3}$.  Gathered monomials such as the latter in which 
no transfer products appear are building blocks for general gathered monomials.

\begin{defn}
A gathered block is a monomial of the form $\prod_{i} {\gamma_{\ell_{i}, n_{i}}}^{d_{i}},$
where the product is the cup product.  Its profile is defined to be the collection
of pairs $(\ell_{i}, d_{i})$.

Non-trivial gathered blocks must have all of the numbers $2^{\ell_{i}} n_{i}$
equal, and we call this number divided by two the width.  
We assume that the factors are ordered from smallest to largest $n_{i}$ (or largest
to smallest $\ell_{i}$), and then note that $n_{i} = 2^{\ell_{1} - \ell_{i}} n_{1}$.
\end{defn}

\begin{prop}
A gathered monomial can be written uniquely as the transfer product of gathered blocks with distinct profiles. 
Gathered monomials form a canonical additive basis for the cohomology of $\coprod_{n} B \si_{n}$.
\end{prop}

Graphically, we represent $\gamma_{\ell,n}$ by a rectangle of width $n \cdot 2^{\ell}$ and height $1 - \frac{1}{2^{\ell}}$, so that its 
area corresponds to its degree. 
We  represent $1_{n}$ by an edge of width $n$ (a height-zero rectangle).
A gathered block, which is a product of $\gamma_{\ell, n}$ for fixed $n \cdot 2^{\ell}$,
is represented by a single column of such rectangles, stacked on top of each other, with order which does not matter.
A gathered monomial is represented by placing such columns next to each other, which we call the skyline diagram of the monomial.  We also refer to the gathered monomial basis as the skyline basis to emphasize
this presentation.

In terms of skyline diagrams, the coproduct can be understood by introducing vertical dashed lines in the
rectangles representing $\gamma_{\ell,n}$, dividing the rectangle into $n$ equal pieces.  The coproduct
is then given by dividing along all existing columns and vertical dashed lines of full height
and then partitioning them into two to make two new skyline diagrams.

The transfer product corresponds to placing two column Skyline diagrams next to each
other and merging columns with the same constituent blocks, with a coefficient of zero if any of those column
widths share a one in their dyadic expansion.  

For cup product, we start with two column diagrams and consider all possible ways to 
split each into columns, along either original boundaries of columns or along the vertical lines of full height
internal to the rectangles representing $\gamma_{\ell,n}$.  We then match columns of each in 
all possible ways up to automorphism, and stack the resulting matched columns to get a new
set of columns, as we saw for $B\si_{4}$.

See Figure~\ref{F:products} for illustrations of cup product, which thus also involve
the coproduct and transfer product.  See Section~6 of \cite{GSS09} for thorough treatments of all of these structures.

\begin{figure}

\begin{eqnarray*}
\raisebox{-.2cm}{\includegraphics[width=2cm]{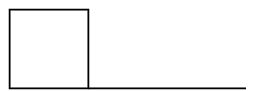}} \cdot \raisebox{-.2cm}{\includegraphics[width=2cm]{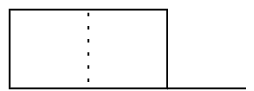} }&=& \raisebox{-.2cm}{\includegraphics[width=2cm]{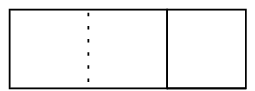} }+ \raisebox{-.2cm}{\includegraphics[width=2cm]{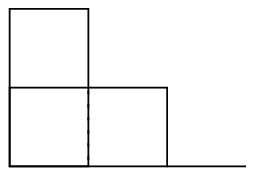}}\\
\gamma_{1,1}\tr 1_4 \cdot \gamma_{1,2}\tr 1_2&=& \gamma_{1,3} + \gamma_{1,1}^2 \tr \gamma_{1,1} \tr 1_2\\
&&\\
\raisebox{-.2cm}{\includegraphics[width=2cm]{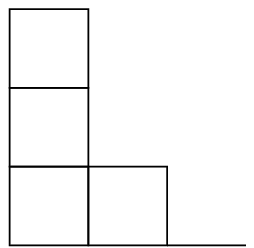}}\cdot\raisebox{-.2cm}{\includegraphics[width=2cm]{images/Mult2}}&= &\raisebox{-.2cm}{\includegraphics[width=2cm]{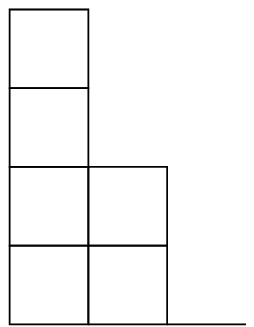}} + \raisebox{-.2cm}{\includegraphics[width=2cm]{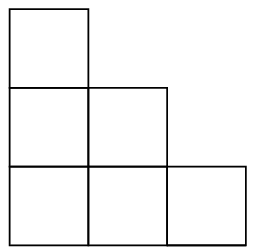}} + \raisebox{-.2cm}{\includegraphics[width=2cm]{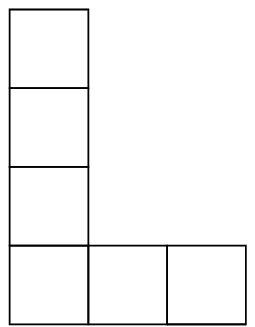}}\\
&=&\raisebox{-.2cm}{\includegraphics[width=2cm]{images/Mult3}} + \raisebox{-.2cm}{\includegraphics[width=2cm]{images/Mult4}}\\
\gamma_{1,1}^3\tr 1_2 \cdot \gamma_{1,2} \tr 1_2 &=& \gamma_{1,1}^4 \tr \gamma_{1,1}^2 \tr 1_2 + \gamma_{1,1}^3 \tr \gamma_{1,1}^2 \tr \gamma_{1,1}\\
&&\\
\raisebox{-.2cm}{\includegraphics[width=2cm]{images/Mult6}} \cdot \raisebox{-.2cm}{\includegraphics[width=2cm]{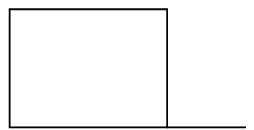}} &=& \raisebox{-.2cm}{\includegraphics[width=2cm]{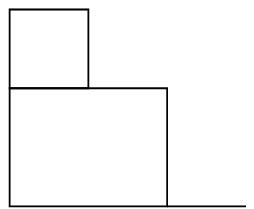}} + \raisebox{-.2cm}{\includegraphics[width=2cm]{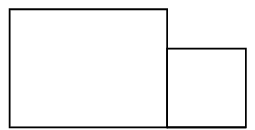}}\\
&=&\raisebox{-.2cm}{\includegraphics[width=2cm]{images/Mult11}}\\
\gamma_{1,1}\tr 1_4 \cdot \gamma_{2,1}\tr 1_2 &=&\gamma_{2,1} \tr \gamma_{1,1}
\end{eqnarray*}

\caption{Examples of product computations in $H^*(B\si_6;\Z/2)$ expressed in skyline  and Hopf monomial notation}
\label{F:products}
\end{figure}

We can use our formula for multiplication to see for example that Vassiliev's conjecture that $d$th powers in the cohomology of 
$\UConf{n}{d}$ are zero is not true.  The cohomology of $\UConf{n}{d}$ is the quotient of that of $\UConf{n}{\infty}$ setting skyline
diagrams with some block height greater than or equal to $d$ to zero.  Then for example in $\UConf{6}{3}$ the cube of 
{\includegraphics[width=2cm]{images/Mult6}} is {\includegraphics[width=2cm]{images/Mult8}}.

\section{The cohomology of $B \si_{\infty}$ as an algebra over the Steenrod algebra}

\subsection{Nakaoka's theorem revisited}
The infinite symmetric group plays a special role in algebraic topology.  Let $\Omega^{\infty} S^{\infty}$ denote the
direct limit, under suspension of maps,  
of the space of based maps from $S^{d}$ to itself.  The Barratt-Priddy-Quillen-Segal theorem \cite{BaPr72} says
that the cohomology of $\Omega^{\infty} S^{\infty}$ is isomorphic to that of $B \si_{\infty}$.

The map $B\si_{n} \to B \si_{n+1}$ induced by inclusion induces the map on cohomology which sends a skyline diagram
with at least one empty column to that obtained by removing that column, and is zero on diagrams with no empty columns.
The inverse limit is thus spanned by skyline diagrams with a finite number of non-empty columns, along with infinitely many
empty columns, which we ignore.  We let the width of such a diagram be the total width of the non-empty columns.
Multiplication is through essentially the same algorithm as in for $B \si_{n}$, which generally increases width unless for example
some diagram is raised to a power of two.

\begin{defn}
A column is even if every block type occurs an even number of times, and odd if at least one block type occurs an odd number of times.
Define the two-root of a skyline diagram $D$ consisting of a single column as the odd column $R$ such that $R^{2^{p}} = D$.
\end{defn}

\begin{thm}[after Nakaoka \cite{Naka61}]\label{T:naka}
The mod-two cohomology of $B\si_{\infty}$ is polynomial, generated by diagrams consisting of a single odd column.
\end{thm}

This theorem along with \refT{genandrel} extends Nakaoka's theorem by giving cochain representatives in ${\fn_{\infty}}^{*}$ for generators.
This result should also be compared with the explicit calculation of homology primitives by Wellington \cite{Well82}.

\begin{proof}
We filter the cohomology of $B \si_{\infty}$, as represented by skyline diagrams, by width.  If $D$ is a diagram with columns
$C_{1}, \ldots, C_{n}$, let $R_{1}, \ldots, R_{n}$ be their two-roots with ${R_{i}}^{2^{p_{i}}} = C_{i}$.  By abuse let $R_{i}$ denote
the diagram which consists of $R_{i}$ as its only non-empty column.  Using the algorithm to multiply skyline diagrams by stacking columns,
$D = \prod_{i} {R_{i}}^{2^{p_{i}}}$ modulo diagrams of lower filtration.  Thus the associated graded to the width filtration is a polynomial
algebra on diagrams consisting of a single odd column. So the cohomology of $B \si_{\infty}$ itself must be polynomial.
\end{proof}

In the course of proof, we see that the change of basis between the skyline basis and the monomial basis arising from the
theorem is non-trivial but straightforward.

\subsection{Steenrod structure}
Next we focus on the Steenrod algebra structure on the cohomology of $B \si_{\infty}$ or equivalently
$\Omega^{\infty} S^{\infty}$.  To do so it is best to use the connection between the cohomology of groups and invariant theory
which has been a fundamental tool in the subject.  Recall for example from Chapters~3~and~4
of \cite{AdMi94} that if $H$ is a subgroup of
$G$ then the Weyl group of $H$ in $G$ acts on the cohomology of $H$.  Moreover, the restriction map from the cohomology
of $G$ to that of $H$ has image in the invariants under this action, namely $(H^{*}(BH))^{W(H)}$.

In the study of the cohomology of symmetric groups, the invariant theory which arises is classical.
Let the subgroup $H$ in question be the subgroup
$(\Z/2)^{n} \subset \si_{2^{n}}$ defined by having $(\Z/2)^{n}$ act on itself, which we call $V_{n}$.
The cohomology of $V_{n}$ is that of $\prod_{n} \R P^{\infty}$, namely
$\F_{2}[x_{1}, \ldots, x_{n}].$
If we view the action of $V_{n}$ on itself
as given by linear translations on the $\F_{2}$-vector space
$\oplus_{n}{\F_{2}}$, then we can see that the normalizer of this subgroup
is isomorphic to all affine transformations of $(\F_{2})^{n}$.
The Weyl group is thus $GL_{n}(\F_{2})$.   Moreover, the action is by linear action on the variables
$x_{1}, \ldots, x_{n}$.

The invariants  $\F_{2}[x_{1}, \ldots, x_{n}]^{GL_{n}(\F_{2})}$, studied a century ago, are known as Dickson algebras.  
Because permutation matrices
are in $GL_{n}(\F_{2})$ the invariants are in particular symmetric polynomials.  But for example there is never a $GL_{n}(\F_{2})$
invariant in degree one, since the lone symmetric invariant $x_{1} + \cdots + x_{n}$ is not invariant under the linear substitution
$x_{1} \mapsto x_{1} + x_{2}$ (and $x_{i} \mapsto x_{i}$ for $i>1$).
Dickson's theorem is that as rings these invariants are
polynomial algebras on generators $d_{k,\ell}$ in dimensions $2^{k}(2^{\ell} - 1)$ where $k + \ell = n$.  For example,
$\F_{2}[x_{1},  x_{2}]^{GL_{2}(\F_{2})}$ is generated by an invariant in degree two, namely ${x_{1}}^{2} +
{x_{2}}^{2}  + x_{1} x_{2}$, along with ${x_{1}}^{2}x_{2} + x_{1}{x_{2}}^{2}$ in degree three.

This connection to invariant theory allows us to determine the action of the Steenrod algebra.  The standard
starting point is that of the cohomology of $\R P^{\infty}$, which allows us to understand the Steenrod structure
on that of $B V_{n} \simeq \prod_{n} \R P^{\infty}$ by the Cartan formula.  Because the Steenrod action is defined by squaring
individual variables, which is a linear operation over $\F_{2}$, the $GL_{n}$-invariants are preserved by the Steenrod action.
For example
$$\sq^{1}({x_{1}}^{2} +
{x_{2}}^{2}  + x_{1} x_{2}) = 0 \cdot {x_{1}}^{3} + 0 \cdot {x_{2}}^{2} + {x_{1}}^{2}x_{2} + x_{1} {x_{2}}^{2}.$$
 In \cite{Hung91} Hu'ng calculated the Steenrod squares on Dickson classes as given by
\begin{equation} \label{E:hung}
\sq^i d_{k, \ell}=
\begin{cases}
d_{k', \ell'}  \;\; & i = 2^{k} - 2^{k'}\\
d_{k', \ell'} d_{k'', \ell''}  \;\;   & i = 2^{n} + 2^{k} - 2^{k'} - 2^{k''}, \;\; k' \leq k < k'' \\
{d_{k, \ell}}^{2} \;\; &i = 2^{k}(2^{\ell} - 1) \\
0 \;\;\;\;\;\;\;\; & {\text otherwise}.
\end{cases}
\end{equation}

Turning our attention back to symmetric groups, the transfer product in cohomology is induced by a
stable map.  Thus there is a Cartan formula for transfer product as well, and it suffices
to understand Steenrod structure on Hopf ring generators.
In \cite{GSS09} we prove the following.

\begin{thm}\label{C:maptodickson}
The restriction of $\gamma_{\ell,2^{k}}$ with $k + \ell = n$
to the elementary abelian subgroup $V_{n}$ is the Dickson invariant $d_{k, \ell}$.
\end{thm}

This theorem,  Hu'ng's calculation above and the fact that the direct some of the restriction map to $V_{n}$ and the coproduct $\Delta$
is injective allow us to understand the Steenrod on the $\gamma_{\ell, 2^{k}}$.

\begin{thm}\label{T:steenact}
A Steenrod
square on $\gamma_{\ell,2^{k}}$ is represented in the skyline basis by
the sum of all diagrams which are of full width, with at most two boxes stacked
on top of each other, and with the width of columns delineated by any of the vertical lines (of full height) at least $\ell$.
\end{thm}

See Figure~\ref{F:steen} for some examples using the Cartan formula and this result.

\begin{figure}
\begin{eqnarray*}
\sq^1\raisebox{-.2cm}{\includegraphics[width=2cm]{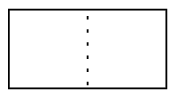}} &=&\raisebox{-.2cm}{\includegraphics[width=2cm]{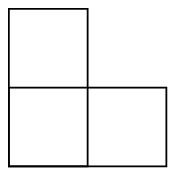}} + \raisebox{-.2cm}{\includegraphics[width=2cm]{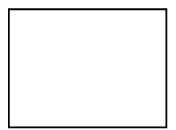}}\\
\sq^1\sq^2\raisebox{-.4cm}{\includegraphics[width=2cm]{images/Steenrod_21}} &=& \sq^1\raisebox{-.4cm}{\includegraphics[width=2cm]{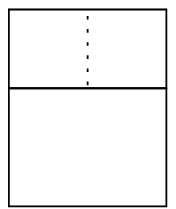}}
=\raisebox{-.4cm}{\includegraphics[width=2cm]{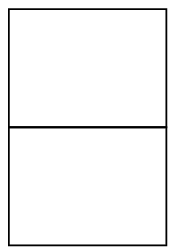}} = \sq^3 \raisebox{-.4cm}{\includegraphics[width=2cm]{images/Steenrod_21}}
\end{eqnarray*}
\caption{Steenrod algebra action on some elements of the skyline basis for $B\si_{4}$.  
The second calculation uses the first through the Cartan formula, with the
term $(\gamma_{1,1}^{2} \odot \gamma_{1,1}) \cdot \gamma_{2,1}$ being zero.  }\label{F:steen}
\end{figure}

\subsection{Madsen's theorem revisited}

One of the first questions to ask about an algebra $R$ over the Steenrod algebra is its vector space of indecomposables,  which
we denote $\indec_{\aalg} R$.  If $R$ is the cohomology of a space,
it is this vector space which can evaluate non-trivially on the Hurewicz homomorphism from homotopy to homology.  
In the case of
the infinite symmetric group (and thus $\Omega^{\infty} S^{\infty}$), this question is reduced to one in invariant theory.
This connection was first made by Madsen forty years ago \cite{Mads75}, 
but there is still much to learn about the resulting algebraic question \cite{HuPe98}.

Our starting point is \refT{naka}, which immediately determines the indecomposables of $H^{*}(B \si_{\infty})$ as an algebra.

\begin{defn}
Define the width splitting of $\indec_{\al} H^{*}(B\si_{\infty})$, the algebra indecomposables of $H^{*}(B \si_{\infty})$,
 by letting $W_{k}$ denote the span of indecomposables which are represented by single
columns of width $2^{k-1}$.  Define the width filtration through the increasing sums $\bigoplus_{k=1}^{n} W_{k}$.
\end{defn}

We have the following remarkable connection between topology and invariant theory, which goes back at least to work of Selick (and Cohen and Peterson) \cite{Seli82}.

\begin{thm}\label{T:dickson1}
The associated graded of the width filtration on $\indec_{\al} H^{*}(B\si_{\infty})$ 
is isomorphic as ${\mathcal{A}}$-modules to the direct sum of Dickson algebras
$\bigoplus_{n} \F_{2}[x_{1}, \ldots, x_{n}]^{GL_{n}(\F_{2})}$.
\end{thm}

\begin{proof}
In the description for Steenrod action on the Hopf ring generators $\gamma_{\ell, 2^{k}}$ of \refT{steenact} there is exactly
one term in which there is just one column.
This term corresponds to the formula for squares on Dickson generators as in Equation~\ref{E:hung} (by replacing
$d_{k, \ell}$ by $\gamma_{\ell, 2^{k}}$).  All other terms have more than one column, and thus modulo decomposables
are lower in the width filtration.
\end{proof}

The fact that for an algebra $R$ over the Steenrod algebra, the quotient map induces an isomorphism of graded vector spaces
$\indec_{\aalg} R \cong \indec_{\amod} (\indec_{\al} R)$ implies that
$\indec_{\aalg} H^{*}(\Omega^{\infty}S^{\infty})$ is a quotient of  $\bigoplus_{n} \indec_{\amod} \F_{2}[x_{1}, \ldots, x_{n}]^{GL_{n}(\F_{2})}.$
To date, the most extensive calculations have been of the dual primitives in homology by Wellington \cite{Well82} and recently by 
Zare \cite{Zare11}.  The Steenrod structure on the algebra indecomposables was put to great use by Campbell, Cohen, Peterson and Selick \cite{CPS86, CCPS86}.  The skyline basis in cohomology gives a distinct approach, which may be especially useful in tandem with homology.

\bibliographystyle{amsplain}
\bibliography{references}

\end{document}